\documentclass[a4paper]{amsart}

\usepackage[english]{babel}

\usepackage{amssymb}
\usepackage{hyperref}
\usepackage{xcolor}
\usepackage{eucal}

\usepackage{amssymb}
\usepackage{amsmath}
\usepackage{amsthm}

\theoremstyle{definition}

\newtheorem{mydef}{Definition}[section]

\theoremstyle{remark}

\newtheorem{mybem}[mydef]{Remark}

\theoremstyle{plain}

\newtheorem{mysen}[mydef]{Theorem}
\newtheorem{mylem}[mydef]{Lemma}
\newtheorem{mypro}[mydef]{Proposition}
\newtheorem{myfact}[mydef]{Fact}

\newtheorem{myclan}[mydef]{Claim}
\newtheorem{mysclai}[mydef]{Subclaim}
\newtheorem{myquest}[mydef]{Question}

\numberwithin{mydef}{section}

\DeclareMathOperator{\dom}{dom}

\DeclareMathOperator{\crit}{crit}

\DeclareMathOperator{\otp}{otp}

\DeclareMathOperator{\Add}{Add}

\DeclareMathOperator{\Odd}{Odd}
\DeclareMathOperator{\Even}{Even}
\DeclareMathOperator{\op}{op}

\DeclareMathOperator{\ICNIA}{\textup{\textsf{ICNIA}}}

\newcommand{\dA}{\mathbb{A}}

\newcommand{\dI}{\mathbb{I}}

\newcommand{\dM}{\mathbb{M}}
\newcommand{\dP}{\mathbb{P}}
\newcommand{\dQ}{\mathbb{Q}}

\newcommand{\dT}{\mathbb{T}}

\newcommand{\uhr}{\upharpoonright}

\newcommand{\seq}[2]{\langle #1 : #2 \rangle}
\newcommand{\COM}{\textsf{\textup{COM}}}
\newcommand{\INC}{\textsf{\textup{INC}}}

\title[Distinguishing Int. Club and Appr. on an Infinite Interval]{Distinguishing Internally Club and Approachable on an Infinite Interval} 

\author{Hannes Jakob and Maxwell Levine} 

\subjclass[2020]{} 

\date{\today}

\begin{document}
	
	
	

	\begin{abstract} Krueger showed that $\textup{\textsf{PFA}}$ implies that for all regular $\Theta \ge \aleph_2$, there are stationarily many $[H(\Theta)]^{\aleph_1}$ that are internally club but not internally approachable. From countably many Mahlo cardinals, we force a model in which, for all positive $n<\omega$ and $\Theta \ge \aleph_{n+1}$, there is a stationary subset of $[H(\Theta)]^{\aleph_n}$ consisting of sets that are internally club but not internally approachable. The theorem is obtained using a new variant of Mitchell forcing. This answers questions of Krueger.
	\end{abstract}
	
	\maketitle
	
	\section{Introduction}

Following work of Foreman and Todor{\v c}evi{\' c} \cite{Foreman-Todorcevic2005}, Krueger wrote a series of papers exploring variations of internal approachability, in particular proving that the variations are distinct \cite{Krueger2007}. He showed that these distinctions can be obtained using mixed-support iterations, which resemble the forcings Mitchell used to obtain the tree property at double successor cardinals. Notable developments in the study of the tree property pertain to obtaining the tree property simultaneously on long intervals of cardinals, and this area of research requires analyses of variants of Mitchell's forcing. In this spirit, Kruger raised the question of whether these properties could be separated for successive cardinals, or even an infinite sequence of cardinals \cite{Krueger2009}. We studied the case in which internally stationary is distinguished from internally club by using forcings that accomplish the work of mixed support iterations while more explicitly resembling Mitchell's forcing \cite{Levine2023a,Jakob2023}.

In this paper we will demonstrate the robustness of this idea by addressing the separation of internally club from internally approachable. We introduce a new version of Mitchell forcing, for which we must consider somewhat elaborate termspaces. The benefit is derived from having an Abraham-style projection analysis. We hope that this concept will be useful for points in the literature where mixed support iterations are called for (see \cite{Fuchino-Rodrigues2018}, for example).


The concepts we study here are framed in terms of the notion of stationarity for spaces of the form $[X]^{\le \mu}$, which was formulated by Jech (see \cite{Jech2003}). We say that some $N \in [X]^\mu$ is:

\begin{itemize}
\item \emph{internally unbounded} if $[N]^{<\mu}\cap N$ is unbounded in $[N]^{<\mu}$,
\item \emph{internally stationary} if $[N]^{<\mu} \cap N$ is stationary in $[N]^{<\mu}$,
\item \emph{internally club} if $[N]^{<\mu} \cap N$ contains a club in $[N]^{<\mu}$,
\item \emph{internally approachable} if there is a continuous sequence $\seq{N_i}{i<\mu}$ consisting of elements of $[N]^{<\mu}$ such that for all $i<\mu$, $\seq{N_j}{j \le i} \in N$ and $N = \bigcup_{i<\mu}N_i$.
\end{itemize}

For clarity, let $\ICNIA(\Theta,\mu)$ be the statement that $\Theta \ge \mu^+$ and that there exist stationarily many $N\in[H(\Theta)]^{\leq\mu}$ which are internally club but not internally approachable. Since the assumption that $\mu$ is regular is standard for stationary subsets of $[H(\Theta)]^{\leq \mu}$, this distinction does not make sense if $\mu$ is singular. Furthermore, it cannot hold if $\mu$ is inaccessible, so in all cases we are assuming that $\mu$ is a double successor. Krueger showed that $\textup{\textsf{PFA}}$ implies $\ICNIA(\Theta,\aleph_1)$ for all $\Theta \ge \aleph_2$ \cite{Krueger2007} and later showed that $\ICNIA(\mu^+,\mu)$ is consistent from a Mahlo cardinal for regular $\mu$ \cite{Krueger2009}. We extend that result here:

\begin{mysen}\label{omega-theorem} Assume there are countably many Mahlo cardinals in $V$. Then there is a forcing extension in which, for all $1 \le n < \omega$, $\ICNIA(\Theta,\aleph_n)$ holds for all $\Theta \ge \aleph_{n+1}$.\end{mysen}

This resolves a case of \cite[Question 12.9]{Krueger2009}, where the projection analysis allows us to obtain consecutive instances of $\ICNIA(\Theta,\aleph_n)$. It also resolves a case of \cite[Question 12.7]{Krueger2009}, where the idea of the solution is more or less that the image of the Mostowski collapse cannot distinguish between $H(\Theta)$'s for large $\Theta$.

	\section{The New Forcing}
	
	In this section we will define the new forcing and present a simple application before moving on to the proof of our main theorem.
	
	\subsection{Defining the Forcing}
	
	The idea of our forcing is to take the two-step iteration used to establish $\ICNIA(\Theta,\aleph_1)$ and build it into a variant of Mitchell forcing that enjoys some of the nice properties of the more standard variants. First, we need the collapse that Krueger used, which forces a chain through a stationary set.
	
	\begin{mydef}[see \cite{Krueger2007}]\label{CollapseClub}
		Let $\mu\leq\delta$ be cardinals and $S\subseteq[X]^{<\mu}$ be stationary for some set $X$. $\dP(S)$ consists of closed sequences of length $<\mu$ through $S$, i.e$.$ it consists of sequences $s$ such that $\dom(s)$ is a successor ordinal below $\mu$, $s(\alpha)\in S$ for all $\alpha\in\dom(s)$, and $s(\gamma)=\bigcup_{\alpha<\gamma}s(\alpha)$ for all limit $\gamma\in\dom(s)$. 
	\end{mydef}
	
The poset $\dP(S)$ is used because it allows us to collapse $\delta$ while retaining both the approximation property and the clubness of the ``old'' sets.
	
	\begin{myfact}
		Let $\mu\leq\delta$ be cardinals and $S\subseteq[H(\delta)]^{<\mu}$ stationary.
		\begin{enumerate}
			\item $\dP(S)$ adds an increasing and continuous sequence of elements $\seq{S_i}{i<\mu}$ of $S$ with union $H(\delta)^V$, thus collapsing $\delta$ to have cardinality $\mu$.
			\item If $\delta^{<\delta}=\delta$, then $\dP(S)$ has cardinality $\delta$ and therefore is $\delta^+$-cc.
		\end{enumerate}
	\end{myfact}
	
	Note that to obtain ${<}\,\mu$-distributivity of $\dP(S)$, $S$ needs to satisfy some additional assumptions.
	
	Now we are ready to define our Mitchell forcing. We need to take some care regarding the model in which the Cohen sets are defined. This is analogous to constructions in which the tree property holds on an interval of cardinals (see \cite{Cummings-Foreman1998}), and is done here in anticipation of the iteration used to prove Theorem~\ref{omega-theorem}. We will therefore use the following basic fact from here on without comment:
	
	\begin{myfact} Suppose that $W \subseteq V$ are models of set theory such that $\tau$ is regular and $\kappa$ is inaccessible in in $W$. Suppose also that $\tau$ and $\kappa$ are cardinals in $V$ and that the extension $W \subseteq V$ has the $\kappa$-covering property. Then $\Add^W(\tau,\kappa)$, the version of $\Add(\tau,\kappa)$ defined in $W$, has the $\kappa$-Knaster property in $V$. (See \cite[Lemma 2.6]{Cummings-Foreman1998} and \cite{Abraham1983}.)\end{myfact}
	
	
	
	\begin{mydef}\label{DefM3}
		Let $W \subseteq V$ be models of $\textup{\textsf{ZFC}}$ containing the ordinals such $(\Add(\tau,\kappa))^W$ is $\mu$-Knaster. Let $\tau<\mu<\kappa$ be cardinals in $V$ such that $\tau^{<\tau}=\tau$ and $\kappa$ is inaccessible. Then we define $\dM^\oplus(\tau,\mu,\kappa,W)$ in $V$ to be the poset consists of pairs $(p,q)$ such that:
		\begin{enumerate}
			\item $p\in (\Add(\tau,\kappa))^W$
			\item $q$ is a $<\mu$-sized function on $\kappa$ such that
			\begin{itemize}
				\item[(a)] for each $\alpha\in\dom(q)$, $\alpha=\delta+1$ for an inaccessible cardinal $\delta$,
				\item[(b)] $q(\alpha)$ is an $(\Add(\tau,\alpha))^W$-name for an element in the chain forcing $\dP([H(\delta)]^{<\mu}\cap V[(\Add(\tau,\delta))^W])$.
				\end{itemize}
		\end{enumerate}
		We let $(p',q')\leq_{\dM^\oplus(\tau,\mu,\kappa,W)}(p,q)$ if
		\begin{enumerate}
			\item $p'\leq_{(\Add(\tau,\kappa))^W} p$
			\item $\dom(q')\supseteq\dom(q)$ and for all $\alpha\in\dom(q)$,
			$$p'\uhr\alpha\Vdash q'(\alpha)\leq_{\Add(\tau,\alpha)^W} q(\alpha).$$
		\end{enumerate}
		For simplicity, we define $\dM^{\oplus}(\tau,\mu,\kappa):=\dM^{\oplus}(\tau,\mu,\kappa,V)$.
	\end{mydef}
	
		 Given our definition of the Mitchell forcing, it then becomes clear that we can define a termspace forcing, which is essentially the main benefit of this presentation.

\begin{mydef} Let $\dT=\dT(\dM^\oplus(\tau,\mu,\kappa,W))$ be the termspace of $\dM^\oplus(\tau,\mu,\kappa,W)$. For the sake of explicitness, this is the poset consisting of conditions $q$ such that:

\begin{enumerate}

\item[(2)] $q$ is a $<\mu$-sized function such that for each $\alpha\in\dom(q)$:

\begin{itemize}
\item[(a)] $\alpha=\delta+1$ for an inaccessible cardinal $\delta$,
\item[(b)] $q(\alpha)$ is an $(\Add(\tau,\alpha))^W$-name for an element in the chain forcing $\dP([H(\delta)]^{<\mu}\cap V[(\Add(\tau,\delta))^W])$.
\end{itemize}

\end{enumerate}

Most importantly, we let $q \le q'$ if and only if:

\begin{enumerate}

\item $\dom q \supseteq \dom q'$,

\item for all $\alpha \in \dom q$, $\Vdash_{\dP([H(\delta)]^{<\mu}\cap V[(\Add(\tau,\delta))^W])} ``q(\alpha) \le q'(\alpha)$''.

\end{enumerate}
\end{mydef}



Next we will establish strategic closure properties of our forcing.
	
	\begin{mydef}
		Let $\dP$ be a forcing order, $\delta$ an ordinal. The \emph{completeness game} $G(\dP,\delta)$ on $\dP$ with length $\delta$ has players $\COM$ (complete) and $\INC$ (incomplete) playing elements of $\dP$ with $\COM$ playing at even ordinals (i.e$.$ limit ordinals and ordinals of the form $\alpha+n$ for $\alpha$ a limit and $n<\omega$) and $\INC$ playing at odd ordinals. $\COM$ starts by playing $1_{\dP}$, afterwards $p_{\alpha}$ has to be a lower bound of $(p_{\beta})_{\beta<\alpha}$. $\INC$ wins if either player is unable to play at some point $<\delta$. Otherwise, $\COM$ wins.
		
	A poset $\dP$ is \emph{$\delta$-strategically closed} if $\COM$ has a winning strategy for the game $G(\dP,\delta)$. We say that $\dP$ is \emph{strongly $\delta$-strategically closed} if $\COM$ has a winning strategy for the version of the game where they play at odd ordinals and $\INC$ plays at even ordinals (see \cite{Handbook-Cummings} for background on these definitions).
	\end{mydef}
	
	

		The subtlety here is that, even though $\dP(S)$ is in most cases not $\mu$-strategically closed since it destroys the stationarity of a subset of $[\delta]^{<\mu}$, the term ordering on $\Add(\tau)*\dP([\delta]^{<\mu}\cap V)$ is $\mu$-strongly strategically closed.

	\begin{mylem}\label{StratClosed}
		Let $\tau<\mu<\delta$ be cardinals such that $\tau^{<\tau}=\tau$. Then the term forcing $\dT(\dM^\oplus(\tau,\mu,\kappa,W))$ is strongly $\mu$-strategically closed.
	\end{mylem}

	
	\begin{proof} Since products of strongly $\mu$-strategically closed forcings are strongly $\mu$-strategically closed, it is sufficient to argue that the direct extension ordering on $\Add(\tau)^W *\dP([\delta]^{<\mu}\cap V)$, i.e$.$ the ordering $\le^*$ for which $(p,q) \le^* (p',q')$ holds if and only if $p = p'$ and $q \le q'$, is $\mu$-strongly strategically closed.\footnote{The strategicaly closure of the direct extension ordering was first noticed by Krueger with a different proof \cite{Krueger2008b}.} We will suppress notation for the inner model $W$ in this proof for the sake of readability.

		We give a winning strategy for $\COM$ by describing a play of the game of the form $\seq{(p,q_\gamma)}{\gamma<\mu}$ where $p\in\Add(\tau)$. At any odd stage $\gamma$, $\COM$ will play $\dot{q}_{\gamma}$ such that the following holds:
		\begin{enumerate}
			\item There is $\nu_{\gamma}$ such that $p\Vdash\dom(\dot{q}_{\gamma})=\check{\nu}_{\gamma}+1$
			\item There is $x_{\gamma} \in V$ such that $p\Vdash\dot{q}_{\gamma}(\check{\nu}_{\gamma})=\check{x}_{\gamma}$.
		\end{enumerate}
We will argue both that this choice will be possible at every stage and that this is sufficient to keep the game going.
		
		Suppose that $\gamma$ is a limit ordinal: If $\COM$ has played according to the strategy until $\gamma$, we let $\nu_{\gamma}:=\bigcup\{ \nu_{\alpha}:\alpha<\gamma,\alpha \in \Odd\}$ and $x_{\gamma}:=\bigcup \{x_{\alpha}:\alpha<\gamma,\alpha \in \Odd\}$. Then we can find a lower bound: Let $\dot{q}_{\gamma}$ be a name for a condition with domain $\nu_\gamma+1$ such that $\dot{q}_{\gamma}(\alpha)=\dot{q}_{\beta}(\alpha)$ for some $\beta<\alpha$ whenever $\alpha<\nu_{\gamma}$ and such that $\dot{q}_{\gamma}(\nu_{\gamma})=x_{\gamma}$. In particular, this works because we have made it explicit that $x_\gamma \in V$.
		
		Now assume $\gamma=\beta+1$ is a successor ordinal and $\INC$ has just played $\dot{q}_{\beta}$. Because $\Add(\tau)$ is $\mu$-Knaster and in particular has the $<\mu$-covering property, $\nu_{\gamma}':=\sup\{\nu\;|\;\exists p'\leq p(p'\Vdash\dom(\dot{q}_{\beta})=\check{\nu})\}$ is below $\mu$ and $x_{\gamma}:=\{\epsilon\;|\;\exists p'\leq p(p'\Vdash\check{\epsilon}\in\bigcup\dot{q}_{\beta})\}$ has size $<\mu$. Let $\dot{q}_{\gamma}$ be a function with domain $\nu_{\gamma}'+1$ extending $\dot{q}_\beta$ and such that $\dot{q}_{\gamma}(\nu_{\gamma}')=\check{x}_{\gamma}$.
		
		
		
		We show that $\dot{q}_{\gamma}$ is as required: $\dot{q}_{\gamma}$ is obviously forced to extend $\dot{q}_{\beta}$. Furthermore, the values of $\dot{q}_{\gamma}$ are forced to be elements of $V$: Until $\nu_{\gamma}'$ this holds because $\dot{q}_{\beta}$ is forced to be in $\dP([\delta]^{<\mu}\cap S)$. At $\nu'_\gamma$, it holds because $\dot{q}_{\gamma}(\nu_{\gamma}')$ is the check-name $\check{x}_\gamma$. Lastly, $\dot{q}_{\gamma}$ is continuous at every limit and increasing. Because $\dom(\dot{q}_{\gamma})=\check{\nu_{\gamma}+\alpha+1}$ and $\dot{q}_{\gamma}(\check{\nu}_{\gamma}+\alpha)=\check{x}_{\gamma}$, we are done.
	\end{proof}

	In particular, by Easton's Lemma, $\dP([\delta]^{<\mu}\cap V)$ is ${<}\,\mu$-distributive (actually strongly ${<}\,\mu$-distributive) in $V[\Add(\tau)]$.
	
		We note that what we have given is actually a winning \emph{tactic}, i.e$.$ in successor stages the play by $\COM$ depends only on the last play of $\INC$, not on the plays before that (see \cite{Yoshinobu2017}).

	Now we can show that $\dM^\oplus$ has similar properties to more standard versions of Mitchell forcing:

\begin{mypro} Let $\dM^\oplus = \dM^\oplus(\tau,\mu,\kappa,W)$:

\begin{enumerate}
\item $\dM^\oplus$ is $\kappa$-Knaster,
\item $\dM^\oplus$ is a projection of the product $\Add(\tau,\kappa)^W \times \dT(\dM^\oplus(\tau,\mu,\kappa,W))$,
\item $\dM^\oplus$ forces $\kappa = 2^\tau = \mu^+ = \tau^{++}$.
\end{enumerate}\end{mypro}

\begin{proof}[Sketch] Recall that the first point follows from a $\Delta$-system argument, the second point uses some mixing of forcing names, and the third point uses the first two points along with Easton's Lemma.\end{proof}

The next point will be the crux of what is needed to bring Krueger's arguments into our context.
	
			\begin{mylem}\label{ApproxProp}
$\dM^{\oplus}(\tau,\mu,\kappa,W)$ has the $<\mu$-approximation property.\end{mylem}

	The argument uses the fact that $\dM^{\oplus}(\tau,\mu,\kappa,W)$ is \emph{iteration-like}, meaning that we can mix the conditions in the second coordinate to ``move disagreements into the first coordinate''.\footnote{See \cite{Jakob2023} for a generalization that uses the notion of strong distributivity, due to the first author.} Our argument here uses ideas of Usuba \cite{Usuba2014} and Unger \cite{Unger2015}.

	\begin{proof}[Proof of Lemma~\ref{ApproxProp}]
	
	Let us abbreviate $\dM^{\oplus}(\tau,\mu,\kappa,W)$ as $\dM$.

	\begin{myclan}\label{DecisionByP}
		Suppose that $(p,q) \in \dM$ forces $\dot{x}\in V$ but that there is no $y\in V$ such that $(p,q)$ forces $\dot{x}=\check{y}$. Then there are $q',p_0,p_1 \le p$ and $y_0\neq y_1$ such that $q' \le_\dT q$ and $y_i \in V$ and $(p_i,q')\Vdash\dot{x}=\check{y}_i$ for $i\in 2$.
	\end{myclan}
	
	
	\begin{proof}
		We consider two possible cases:

\emph{Case 1:} There are $q^*$ and $y_0$ such that $(p,q^*) \le_{\dM} (p,q)$ and $(p,q^*)\Vdash\dot{x}=\check{y}_0$. 

Then choose $(p_1,q^{**}) \le_{\dM} (p,q)$ and some $y_1$ such that $(p_1,q^{**}) \Vdash \dot{x} = \check{y}_1$. Strengthen if necessary to assume that $p_1$ is strictly below $p$ and choose $p_0 \le p$ incompatible with $p_1$. Using standard arguments for the construction of names, there is $q'$ such that $q' \le_\dT q$ and such that for all $\alpha \in \dom q^{**}$, $p_1 \Vdash q^{**}(\alpha)=q'(\alpha)$ and for all $\alpha \in \dom q^*$, $p_0 \Vdash q^*(\alpha) = q'(\alpha)$. Then we have $(p_0,q') \le_{\dM} (p_0,q^*) \le_{\dM} (p,q^*) \le_{\dM} (p,q)$ and $(p_1,q') \le_{\dM} (p_1,q^{**}) \le_{\dM} (p,q)$, and so have this case of the claim.

\emph{Case 2:} For all $q^*$ with $(p,q^*)\le_{\dM}(p,q)$, $(p,q^*)\not\Vdash\dot{x}=\check{y}_0$ for any $y_0$. 

Then choose $(p_0,q^*) \le_{\dM} (p,q)$ forcing $\dot{x}=\check{y}_0$ for some $y_0$. Using the mixing of names, we can assume that $q^* \le_{\dT} q$, and hence that $(p,q^*) \le_{\dM} (p,q)$. The present case implies that $(p,q^*) \not\Vdash\dot{x}=\check{y}_0$, so there is some $(p_1,q') \le_{\dM} (p,q^*)$ forcing $\dot{x}=y_1$ for some $y_1 \ne y_0$. Again we can assume that $q' \le_{\dT} q^*$. Therefore $(p_0,q') \le_{\dM} (p_0,q^*) \le_{\dM} (p,q^*) \le_{\dM} (p,q)$ and $(p_1,q') \le_{\dM} (p,q^*) \le (p,q)$.\end{proof}	
	
%
%
%
%

		Now suppose for contradiction that the lemma is false. Let $\dot{f}$ be an $\dM$-name such that some $(p,q)$ forces every $<\mu$-approximation to be in $V$, but $\dot{f}$ itself to be outside of $V$. For simplicity, assume $(p,q)=1_{\dM}$.
		
	We will use the winning strategy for $\COM$ in the completeness game of length $\mu$ played on $\dT$. More precisely, the values of $q_\gamma$ chosen for $\gamma \in \Even$ are chosen by $\INC$, and the construction continues because of the winning strategy for $\COM$.
	
	We will construct $(p_{\gamma}^0,p_{\gamma}^1,q_\gamma,y_{\gamma})_{\gamma\in\Even}$ such that

		\begin{enumerate}
			\item $y_{\gamma}\in [V]^{<\mu}\cap V$, the sequence $(y_{\gamma})_{\gamma\in\Even}$ is $\subseteq$-increasing,
			\item the $q_\gamma$'s are $\le_\dT$-decreasing,
			\item $(p_{\gamma}^0,q_{\gamma})$ and $(p_{\gamma}^1,q_{\gamma})$ decide $\dot{f}\uhr \check{y}_{\alpha}$ the same way for $\alpha<\gamma$, but differently for $\alpha=\gamma$.
		\end{enumerate}
		Assume the game has been played until some even ordinal $\gamma<\mu$. Let $y_{\gamma+1}':=\bigcup_{\alpha<\gamma}y_{\gamma}$, which has size $<\mu$. Because $\dot{f}$ is forced not to be in $V$, there is $y_{\gamma+1}\supseteq y_{\gamma+1}'$ of size $<\mu$ such that $(1_{\dA},q_{\gamma})$ does not decide $\dot{f}\uhr\check{y}_{\gamma+1}$. Thus, we find all required objects by appealing to Claim~\ref{DecisionByP}. Formally, we can choose the plitting below $(p_\gamma^0,q_\gamma)$ at every step.

		We claim that $\{(p_{\gamma}^0,p_{\gamma}^1)\;|\;\gamma\in\Even\}$ is an antichain in $\dA \times \dA$ where $\dA:=\Add(\tau,\kappa)^W$, obtaining a contradiction since it is known that $\dA \times \dA$ has the $\mu$-chain condition. To this end, assume $(p^0,p^1) \le_{\dA \times \dA} (p_{\gamma}^0,p_{\gamma}^1),(p_{\gamma'}^0,p_{\gamma'}^1)$ with $\gamma>\gamma'$. Because $p^0 \le_{\dA} p_{\gamma'}^0$ and $p^1 \le_{\dA} p_{\gamma'}^1$, $(p^0,q_{\gamma}) \le_{\dM} (p^0,q_{\gamma'}) \le_{\dM} (p_{\gamma'}^0,q_{\gamma'})$ and $(p^1,q_{\gamma}) \le_{\dM} (p^1,q_{\gamma'}) \le_{\dM} (p_{\gamma'}^1,q_{\gamma'})$ decide $\dot{f}\uhr\check{y}_{\gamma'}$ differently, but because $p^0 \le_{\dA} p_{\gamma}^0$ and $p^1 \le_{\dA} p_{\gamma}^1$, $(p^0,q_{\gamma}) \le_{\dM} (p_{\gamma}^0,q_{\gamma})$ and $(p^1,q_{\gamma}) \le_{\dM} (p_{\gamma}^1, q_{\gamma})$ decide $\dot{f}\uhr\check{y}_{\gamma'}$ the same way, a contradiction.
	\end{proof}

	
	\subsection{Considering Quotients}

	As is common when working with variants of Mitchell forcing, we give an explicit description of the quotient forcing. In this subsection we will define and state what we need in order to carry out the proof of our main theorem, leaving out some of the details that are addressed elsewhere in the literature.	
	
	\begin{mydef}
		Let $\tau<\mu<\nu<\kappa$ be cardinals such that $\Add(\tau,\kappa)^W$ is $\mu$-Knaster. Let $G\subseteq\dM^\oplus(\tau,\mu,\nu,W)$ be a generic filter. In $V[G]$, define $\dM^\oplus(G,\tau,\mu,\kappa\smallsetminus\nu,W)$ to consist of $(p,q)$ such that
		\begin{enumerate}
			\item $p\in\Add(\tau,\kappa\smallsetminus\nu)^W$
			\item $q$ is a partial function on $\kappa\smallsetminus\nu$ of size $<\mu$ such that for each $\alpha\in\dom(q)$, $\alpha=\nu+1$ for an inaccessible cardinal $\nu$ and $q(\alpha)$ is an $\Add(\tau,\alpha\smallsetminus\nu)$-name for an element of $\dP([\nu]^{<\mu}\cap V[G][\Add(\tau,\nu\smallsetminus\nu)])$.
		\end{enumerate}
		We let $(p',q')\leq(p,q)$ if
		\begin{enumerate}
			\item $p'\leq p$
			\item $\dom(q')\supseteq\dom(q)$ and for all $\alpha\in\dom(q)$,
			$$p'\uhr\alpha\Vdash q'(\alpha)\leq q(\alpha).$$
		\end{enumerate}
	\end{mydef}
	
	We remark that we technically do not need the generic $G$ to define the quotient.
	
	The next lemma follows similarly to other known variants of Mitchell Forcing.
	
	\begin{mylem}\label{M3Decom}
		Let $\tau<\mu<\nu<\kappa$ be regular cardinals such that $\tau^{<\tau}=\tau$ and $\nu,\kappa$ are inaccessible. There is a dense embedding from $\dM^\oplus(\tau,\mu,\kappa,W)$ into $\dM^\oplus(\tau,\mu,\nu,W)*\dM^\oplus(G,\tau,\mu,\kappa\smallsetminus\nu,W)$.
	\end{mylem}
	
	\begin{proof}
		As in other versions of Mitchell Forcing, we define
		$$(p,q)\mapsto(p\uhr\nu,q\uhr\nu,\op(\check{p\uhr(\kappa\smallsetminus\nu)},\overline{q}))$$
		where $\overline{q}$ reimagines $q\uhr(\kappa\smallsetminus\nu)$ as an $\dM^\oplus(\tau,\mu,\nu)$-name.
	\end{proof}
	
	Similarly, we have the following:
	
	\begin{mypro}\label{little-factors} Let $\tau<\mu<\nu<\kappa$ be cardinals such that $\tau^{<\tau}=\tau$ and $\nu,\kappa$ are inaccessible and let $G$ be $\dM^\oplus(\tau,\mu,\nu,W)$-generic over and let $H$ be $\dM^\oplus(G,\tau,\mu,\kappa \setminus \nu,W)$-generic over $V[G]$. Then there is a filter $K_A$ that is $\Add(\tau)$-generic over $V[G]$ and a filter $K_C$ that is $\dP([\nu]^{<\mu} \cap V[G])$-generic over $V[G][K_A]$ such that $V[G][H]$ is a forcing extension of $V[G][K_A][K_C]$.\end{mypro}


\begin{proof} Use a map similar to the one from the previous lemma. This is where we use the fact that $\dom(q)$ consists of ordinal successors of inaccessibles for $(p,q) \in \dM^\oplus(\tau,\mu,\kappa,W)$. Here we also note that $\dP([\nu]^{<\mu} \cap V[G]) = \dP([\nu]^{<\mu} \cap V[A])$ where $A$ is the $\Add(\tau,\nu)$-generic induced by $G$.\end{proof}
	
	In $V[G]$, $\dM^\oplus(G,\tau,\mu,\kappa\smallsetminus\nu,W)$ has similar properties to $\dM^\oplus(\tau,\mu,\kappa,W)$ using arguments similar to the ones we detailed:
	
	\begin{mylem}\label{M3DecProp}
		Let $\tau<\mu<\nu<\kappa$ be cardinals such that $\Add(\tau,\kappa)^W$ is $\mu$-Knaster and $\nu,\kappa$ are inaccessible. Let $G$ be $\dM^\oplus(\tau,\mu,\nu,W)$-generic. The following holds in $V[G]$:
		\begin{enumerate}
			\item $\dM^\oplus(G,\tau,\mu,\kappa\smallsetminus\nu,W)$ is $\kappa$-Knaster.
			\item The term ordering on $\dM^\oplus(G,\tau,\mu,\kappa\smallsetminus\nu,W)$ is $\mu$-strongly strategically closed.
		\end{enumerate}
	\end{mylem}
	
	

It is crucial for us to obtain the approximation property for quotients, which we can obtain from trivial modifications of the proof of Lemma~\autoref{ApproxProp}.

	\begin{mylem}\label{quotient-approx}
		Let $\tau<\mu<\nu<\kappa$ be cardinals such that $\tau^{<\tau}=\tau$ and $\nu,\kappa$ are inaccessible. Let $G$ be $\dM^\oplus(\tau,\mu,\nu,W)$-generic. In $V[G]$, $\dM^\oplus(G,\tau,\mu,\kappa\smallsetminus\nu,W)$ has the $<\mu$-approximation property.
	\end{mylem}
	

\subsection{Distinguishing Internally Club and Approachable for a Single Cardinal}
	
	In this subsection we show that $\dM^\oplus(\tau,\mu,\kappa,W)$ forces $\ICNIA(\kappa,\mu)$ to hold at $\kappa=\mu^+$. Technically, the next theorem will become redundant after giving the proof of Lemma \ref{Extension}. However, the proof serves as a gentle introduction to these arguments.

	\begin{mydef}\cite{Harrington-Shelah1985} Let $K$ be a model of some fragment of $\textup{\textsf{ZFC}}$. We say that $M \prec K$ is \emph{rich} or \emph{rich with respect to $\kappa$} if the following hold: 

\begin{enumerate}
\item $\kappa \in M$;
\item $\bar{\kappa}:=M \cap \kappa \in \kappa$ and $\bar{\kappa}<\kappa$;
\item $\bar \kappa$ is an inaccessible cardinal in $K$;
\item The cardinality of $M$ is $\bar{\kappa}$;
\item $M$ is closed under $<\!\bar{\kappa}$-sequences.
\end{enumerate}
	 \end{mydef}
	 
It is easy to show that:	 
	 
	 \begin{myfact}\label{wegotrichmodels} If $\kappa$ is Mahlo and $K$ is a model of a sufficiently rich fragment of $\textup{\textup{ZFC}}$ with $\kappa+1 \subseteq K$, then for all $a \in [K]^{<\kappa}$, there is a model $M \prec K$ such that $a \subseteq M$ and $M$ is rich with respect to $\kappa$. 
	 \end{myfact} 
	
	\begin{mysen}\label{M3ForcesDist}
		$\dM^\oplus(\tau,\mu,\kappa)$ forces that there exist stationarily many $N\in[H(\kappa)]^{\leq\mu}$ such that $N$ is internally club but not internally approachable.
	\end{mysen}
	
	\begin{proof}
		To aid in the legibility, we let $\dM^\oplus:=\dM^\oplus(\tau,\mu,\kappa)$.
		
		Let $\dot{C}$ be an $\dM^\oplus$-name for a club. Let $\dot{F}$ be an $\dM^\oplus$-name for a function $[H(\kappa)]^{<\omega}\to[H(\kappa)]^{<\kappa}$ such that the closure points of $\dot{F}$ are contained in $\dot{C}$. Let $\Theta$ be a cardinal such that $\dot{F}\in H(\Theta)$ and let $M'\prec H(\Theta)$ be rich with respect to $\kappa$ such that $\dot{F},\dM^\oplus,\tau,\mu,\kappa\in M'$, $\mu\subseteq M'$.

		Let $G$ be $\dM^\oplus$-generic and consider $M'[G]$: Since $\dot{F}^G\in M'[G]$, $M'[G]\cap H(\kappa)$ is closed under $\dot{F}^G$ and thus $M'[G]\cap H(\kappa)\in\dot{C}^G$. Furthermore, because $\dM^\oplus$ is $\kappa$-cc., $M'[G]\cap H(\kappa)=(M'\cap H(\kappa)^V)[G]=:M[G]$. We will show that $M[G]$ is as required.
		
		Let $\pi:M\to N$ be the Mostowski-Collapse of $M$. Because $\dM^\oplus$ is $\kappa$-cc., $M[G]\cap V=M$ and thus $\pi$ extends to $\pi:M[G]\to N[G']$, where $G':=\pi[G]$. Looking at the proof of Lemma \ref{M3Decom}, $G'$ is also equal to the $\dM^\oplus(\tau,\mu,\nu)$-generic filter induced by $G$ and there is an $\dM^\oplus(G',\tau,\mu,\kappa\smallsetminus\nu)$-generic filter (over $V[G']$) $G''$ such that $G=G'*G''$.

		\begin{myclan}
			$M[G]$ is internally club.
		\end{myclan}
		
		\begin{proof}
			We first show that $N[G']$ is internally club. We have $N[G']\subseteq V[G']$. Additionally, the reverse inclusion holds for many sets:
			\begin{mysclai}
				If $x\in [N[G']]^{<\mu}\cap V[G']$, $x\in N[G']$.
			\end{mysclai}
			\begin{proof}
				If $x\in [N[G']]^{<\mu}\cap V[G']$, $x$ has been added by $\Add(\tau,\nu)$. Let $\dot{x}$ be an $\Add(\tau,\nu)$-name for $x$. By the $\tau^+$-cc. of $\Add(\tau,\nu)$, we can assume that $\dot{x}$ is a ${<}\,\mu$-sized subset of $N$ (since $\dot{x}(\alpha)$ is an element of $N[G']$ for every $\alpha$). Then $\dot{x}\in N$ and thus $\dot{x}^{G'}\in N[G']$.
			\end{proof}
			\begin{mysclai}
				$N[G']$ is internally club.
			\end{mysclai}
			\begin{proof}
				By the previous claim, $[N[G']]^{<\mu}\cap V[G']=[N[G']]^{<\mu}\cap N[G']$. $\dM^\oplus(\nu+2)$ collapses $\nu$ by adding a continuous and cofinal sequence into $[\nu]^{<\mu}\cap V[\Add(\tau,\nu)]$ by Proposition~\autoref{little-factors}. This is isomorphic to $[N[G']]^{<\mu}\cap V[G']$ since $|N[G']|=|N|=\nu$. Hence, $\dM^\oplus(\nu+2)$ forces that we can write $N[G']=\bigcup_{i<\mu}N_i$ where $N_i\in [N[G']]^{<\mu}\cap V[G']=[N[G']]^{<\mu}\cap N[G']$ for every $i<\mu$.
			\end{proof}
			Since $\pi$ is an ``internal'' isomorphism of $M[G]$ and $N[G']$, $M$ is also internally club: Write $N[G']=\bigcup_{i<\mu}N_i$ such that $N_i\in [N[G']]^{<\mu}\cap N[G']$ for every $i$. Then $M[G]=\bigcup_{i<\mu}\pi^{-1}[N_i]$ and $\pi^{-1}[N_i]=\pi^{-1}(N_i)\in [M[G]]^{<\mu}\cap M[G]$ for every $i$ (since $\otp(N_i)<\mu<\crit(\pi^{-1})$).
		\end{proof}
		
		Thus we are finished after showing:
		
		\begin{myclan}
			$M[G]$ is not internally approachable.
		\end{myclan}
		\begin{proof}
			Again, we show the following first:
			\begin{mysclai}
				$N[G']$ is not internally approachable.
			\end{mysclai}
			\begin{proof}
				Assume toward a contradiction that $N[G']=\bigcup_{i<\mu}N_i$ such that for each $j$, $(N_i)_{i<j}\in N[G']$. In particular, $(N_i)_{i<j}\in V[G']$. Because $G=G'*G''$ and $G''$ is generic for an ordering with the $<\mu$-approximation property (\autoref{quotient-approx}), $(N_i)_{i<\mu}\in V[G']$. However, this implies that $N[G']$ has size $\mu$ in $V[G']$, a contradiction as $\dM^\oplus(\tau,\mu,\nu)$ is $\nu$-cc. and $|N[G']|=|N|=|\nu|$.
			\end{proof}
			Now assume $M[G]=\bigcup_{i<\mu}M_i$ such that for each $j<\mu$, $(M_i)_{i<j}\in M[G]$. Then $N[G']=\bigcup_{i<\mu}\pi[N_i]$ and for each $j<\mu$, $(\pi[N_i])_{i<j}=(\pi(N_i))_{i<j}=\pi((N_i)_{i<j})\in N[G']$, since $\pi(\mu)=\mu$, giving us a contradiction.
		\end{proof}
		Thus we have produced a set in $\dot{C}^G$ which is internally club but not internally approachable.
	\end{proof}
	
	\section{Distinguishing Internally Club and Approachable on an Infinite Interval}
	
	In this section we apply the previous results to obtain the distinction between internally club and approachable on the interval $[\aleph_2,\aleph_{\omega})$, thus obtaining our main theorem.
	
	\subsection{Preservation of the Distinction}
	
	First we do some preliminary work by establishing some conditions under which $\ICNIA(\Theta,\aleph_n)$ is preserved by sufficiently well-behaved forcings.

	To obtain the model for \autoref{omega-theorem}, we will make use of a projection analysis, showing that, for a given $n$, the distinction holds in an outer model of the target model. With this intention, we introduce a slight strengthening of $\ICNIA$ which is more easily preserved downwards.
	
	
	\begin{mydef}\label{plus-version}
		Let $\ICNIA^+(\Theta,\mu)$ be the statement that $\Theta\geq\mu^+$ and there exist stationarily many $N\in[H(\Theta)]^{\leq\mu}$ such that
		\begin{enumerate}
			\item $N$ is internally club.
			\item There is no sequence $(X_i)_{i<\mu}$ of elements of $[\Theta]^{<\mu}$ such that $\bigcup_{i<\mu}X_i=N\cap \Theta$ and $(X_i)_{i<j}\in N$ for all $j<\mu$.
		\end{enumerate}
		
		 We say that \emph{$N$ is not ordinal-internally approachable} if clause (2) holds.
	\end{mydef}
	
	We easily see that $\ICNIA^+(\Theta,\mu)$ implies $\ICNIA(\Theta,\mu)$: if $N$ is internally approachable, simply intersect the approaching sequence with the class of ordinals.
	
	\begin{mypro}\label{ClubApprDown}
		Assume $W$ is a forcing extension of $V$ by a forcing order $\dP$ which is ${<}\,\mu^+$-distributive. If $\ICNIA^+(\Theta,\mu)$ holds in $W$, $\ICNIA^+(\Theta,\mu)$ holds in $V$.
	\end{mypro}
	
	\begin{proof}
		In $V$, let $C$ be club in $[H(\Theta)^V]^{\leq\mu}$. Then in $W:=V[G]$, $C$ is club in $[H(\Theta)^V]^{\leq\mu}$ by the distributivity. Let $\Theta'$ be larger than $\Theta$ and at least so large that $\dP\in H(\Theta')$. We have the following statement whose form connects it to notions of properness:
		\begin{myclan}
			In $V[G]$, the set
			$$D':=\{M\in[H(\Theta')^V]^{\mu}\;|\;M[G]\cap V=M\}$$
			is club in $[H(\Theta')^V]^{\mu}$.
		\end{myclan}
		\begin{proof}
			For closure, notice that
			$$\left(\bigcup_{i<\mu}M_i\right)[G]\cap V=\left(\bigcup_{i<\mu}M_i[G]\right)\cap V=\bigcup_{i<\mu}(M_i[G]\cap V).$$
			For unboundedness, let $M_0\in[H(\Theta')^V]^{\mu}$ be arbitrary. Inductively define $M_{n+1}:=M_n\cup (M_n[G]\cap V)$. Then
			$$\left(\bigcup_{n\in\omega}M_n\right)[G]\cap V=\bigcup_{n\in\omega}M_n$$
			since, given some $\tau\in M_n$ with $\tau^G\in V$, $\tau^G\in M_n[G]\cap V=M_{n+1}$.
		\end{proof}
		Additionally, $C':=\{M\in[H(\Theta')^V]^{\mu}\;|\;M\cap H(\Theta)^V \in C\}$ is club in $[H(\Theta')^V]^{\mu}$. Thus
		$$E':=\{M[G]\;|\;M\in D'\cap C'\}$$
		is club in $[H(\Theta')^V[G]]^{\mu}$ which equals $[H(\Theta')^W]^{\mu}$ by the size of $\Theta'$. This implies that the set
		$$E:=\{M[G]\cap H(\Theta)^W\;|\;M[G]\in E'\}$$
		contains a club in $[H(\Theta)^W]^{\mu}$.
		
		Thus there exists $M\in D'\cap C'$ such that $M[G]\cap H(\Theta)^W$ is internally club but $\mu^+$ is not ordinal internally approachable in $M[G]\cap H(\Theta)^W$. We aim to show that, in $V$, $M\cap H(\Theta)^V$ is internally club but $\mu^+$ is not approachable in $M\cap H(\Theta)^V$.
		
		\begin{myclan}
			$M\cap H^V(\Theta)$ is internally club in the model $V$.
		\end{myclan}
		
		\begin{proof}
			We can write $M[G]\cap H^W(\Theta)=\bigcup_{i<\mu}M_i$, where the union is continuous and increasing and each $M_i$ is in $[M[G]\cap H(\Theta)]^{<\mu}\cap M[G]\cap H(\Theta)$. Because $M\in D'$, $M=M[G]\cap V$, so
\begin{align*}
M\cap H^V(\Theta)= & (M[G]\cap V)\cap H^W(\Theta)=(M[G]\cap H^W(\Theta))\cap V = \\ = & \bigcup_{i<\mu}M_i\cap V=\bigcup_{i<\mu}M_i\cap H^V(\Theta),
\end{align*}
using that $H^V(\Theta)=H(\Theta)\cap V$ as $\dP$ does not collapse cardinals. As $M[G]\cap V=M$, $M_i\cap H^V(\Theta')=M_i\cap H^V(\Theta)\in M[G]$ for every $i<\mu$. Additionally, $M_i\cap H^V(\Theta)$ is a subset of $H^V(\Theta)$ of size ${<}\,\mu$, so $M_i\cap H^V(\Theta)\in H^V(\Theta)$: $M_i\in V$ by distributivity and has hereditary size ${<}\,\Theta$. Again, as $M\in D$, $M_i\cap H^V(\Theta)\in M[G]\cap V=M$, so in summary $M_i\cap H^V(\Theta)\in M\cap H^V(\Theta)$.
		\end{proof}
		
		\begin{myclan}
			$M\cap H^V(\Theta)$ is not ordinal-internally approachable in the model $V$.
		\end{myclan}
		
		\begin{proof}
			Since $M[G]\cap V=M$, $(M[G]\cap H^W(\Theta))\cap\Theta=(M\cap H^V(\Theta))\cap\Theta$. Thus, if $M\cap H^V(\Theta)$ were ordinal-internally approachable in the model $V$, the same would be the case in the model $W$ (witnessed by the same sequence), a contradiction.
		\end{proof}
		Thus we have produced an element of $C$ which is as required.
	\end{proof}

	\subsection{Proving the Main Theorem}
	
	Now we will set up the proof of Theorem~\autoref{omega-theorem}. Let $(\kappa_n)_{n\in\omega}$ be a sequence of Mahlo cardinals. We force with the full support iteration $\dI = \seq{\dP_n}{n<\omega}$ where $\dP_0 = \dM^\oplus(\omega,\omega_1,\kappa_0,V)$ and given $\dP_n$ we let
	\[
	\dP_{n+1} = \dP_n \ast \dot{\dM}^\oplus(\kappa_{n-1},\kappa_0,\kappa_{n+1},V[\dP_{n-1}])
\]	
where $\kappa_{-2}=\omega$, $\kappa_{-1}=\omega_1$, and $V[\dP_{-1}]=V$ for simplicity. Observe that the iteration will turn $\kappa_n$ into $\aleph_{n+2}$ for all $0 \le n<\omega$.

We start with a small improvement of Theorem \ref{M3ForcesDist}:

		\begin{mylem}\label{Extension}
		Let $\tau<\mu<\kappa$ be cardinals such that $\tau^{<\tau}=\tau$, $\mu=\mu^{<\mu}$ and $\kappa$ is Mahlo. If $\gamma$ is any ordinal, $\dM^\oplus(\tau,\mu,\kappa,W)\times\Add(\mu,\gamma)$ forces $\ICNIA^+(\Theta,\mu)$ for all regular $\Theta \ge \kappa$.
	\end{mylem}

	
	
	\begin{proof}[Proof of Lemma~\ref{Extension}]
		We modify the proof of Theorem \ref{M3ForcesDist}.
		
		Define $\dQ:=\dM^\oplus(\tau,\mu,\kappa,W)\times\Add(\mu,\gamma)$. We will abbreviate this product as $\dM \times \dA_\textup{big}$. We will write the product that projects onto $\dM$ as $\dT \times \dA_\textup{small}$.
		
		 Let $\dot{C}$ be a $\dQ$-name for a club in $[H(\Theta)]^{\leq\mu}$ and $\dot{F}$ a name for the corresponding function. Let $\dot{X}$ be the $\dM \times \dA_\textup{big}$-name for $H(\Theta)^{V[\dM \times \dA_\textup{big}]}$. Suppose for contradiction that there is a condition $\tilde{q} \in \dQ$ forcing that $\dot{C}$ avoids the set of elements in $[H(\Theta)]^{\leq\mu}$ that are internally club and in which $\dot{X}$ is not ordinal internally approachable. (This formulation is necessary because we are using Mahlo embeddings.) Let $\Theta'>\Theta$ be such that $H(\Theta')$ contains $\dot{F}$ and choose a rich model $M\prec H(\Theta)$ with respect to $\kappa$ of cardinality $\nu$ such that $M$ contains $\tilde{q}$, $\dQ$ and $\dot{F}$.
		 
		 Let $\bar{\dM} = \dM \cap M = \pi_M(\dM)= \dM(\tau,\mu,\nu,W)$. Let $\bar{\mathbb{A}}_\textup{big} = \mathbb{A}_\textup{big} \cap M = \pi_M(\dM)$.
		 
		Now we will argue that we can choose the generics in a way that will suit us. Let $G' \times H'$ be a $\bar{\dM} \times \bar{\dA}_\textup{big}$-generic filter containing $\tilde{q}$. We can find $G''_1 \times G_2''$, a product of filters that are $\dT \times \mathbb{A}_\textup{small}$-generic over $V[G']$ such that $\dT$ induces a generic $G''_0$ for $\dM/G$ using Lemma~\ref{M3DecProp}. Now we let $G = G' \ast G''$. Since $\mathbb{A}_\textup{big} = \pi_M(\mathbb{A}_\textup{big}) \times \mathbb{A}'_\textup{big}$ where $\mathbb{A}'_\textup{big}$ is a remainder, there is $H''$ such that $H=H' \times H''$ is $\mathbb{A}_\textup{big}$-generic and $\pi_M(G \times H) = G' \times H'$. (This can be formulated in terms of $j_M$, the reverse of the Mostowski collapse $\pi_M$, and applying Silver's classical lifting criterion.)
		 
		  
		
		We will argue that $(M \cap H(\Theta))[G][H]$ is internally club and $X:=\dot{X}^{G \times H}$ is not internally approachable in $(M\cap H(\Theta))[G][H]$ in the model $V[G][H]$. Since $(M \cap H(\Theta))[G][H]$ is closed under $\dot{F}^{G\times H}$, this suffices. We will argue using $N$, the image $\pi_M:M \to N$.

		
		
		\begin{myclan}
			$(M \cap H(\Theta))[G][H]$ is internally club.
		\end{myclan}
		\begin{proof}
			This holds as in the proof of Theorem \ref{M3ForcesDist}: The Mostowski-Collapse of $(M\cap H(\Theta))[G][H]$ is equal to $\pi(H(\Theta))[G'][H']$ which is closed under ${<}\,\nu$-sequences in $V[G'][H']$. As before, $G\times H$ adds a club in $[\pi(H(\Theta))[G'][H']]^{<\mu}$ consisting of elements of $\pi(H(\Theta))[G'][H']$.
		\end{proof}
			

		The slightly harder claim is:
		
		\begin{myclan}
			$(M\cap H(\Theta))[G][H]$ is not ordinal-internally approachable.
		\end{myclan}
		
		\begin{proof}
			Assume towards a contradiction that there is a sequence $(X_i)_{i<\mu}$ of elements of $[\Theta]^{<\mu}$ such that $(X_i)_{i<j}\in (M\cap H(\Theta))[G][H]$ for every $j<\mu$ and $\bigcup_{i<\mu}X_i=(M\cap H(\Theta))[G][H]\cap\Theta=\nu$. It follows that, for every $j<\mu$, $\pi((X_i)_{i<j})=(\pi[X_i])_{i<j}=(X_i)_{i<j}\in N[G'][H']\subseteq V[G'][H']\subseteq V[G'][H]$. However, $V[G][H]$ is an extension of $V[G'][H]$ by $\dM(G',\tau,\mu,\kappa\smallsetminus\nu)$ which has the ${<}\,\mu$-approximation property in $V[G'][H]$: one easily checks that the proof of Lemma~\autoref{quotient-approx} still works because $\Add(\mu,\gamma)$ is ${<}\,\mu$-distributive and therefore does not change the definition of $\dM(G',\tau,\mu,\kappa\smallsetminus\nu)$. Hence $(X_i)_{i<\mu}\in V[G'][H]$. This implies that $\Theta\geq\nu$ has size $\mu$ in $V[G'][H]$, a contradiction, as $G'\times H$ is generic for a $\nu$-Knaster forcing. 
		\end{proof}

		Again, we have produced $(M \cap H(\Theta))[G][H]\in\dot{C}^{G \times H}$ which is internally club but not internally approachable. This contradicts the choice of $\tilde{q}$.\end{proof}
	

%
	

Now we can finish the proof of Theorem~\autoref{omega-theorem}: Let $n\in\omega$ be arbitrary.

		To obtain $\ICNIA(\Theta,\kappa_{n-1})$ we will view the iteration as a factorization $\dP_\textup{low}^n\ast \dot{\dP}_\textup{next}^n \ast \dot{\dP}_\textup{high}^n$, where
	

\begin{itemize}

\item $\dP_\textup{low}^n :=\dP_{n-1}$,

\item $\dot{\dP}_\textup{next}^n$ is a $\dP_{n-1}$-name for
\begin{align*}\dM^{\oplus}(\kappa_{n-2},\kappa_{n-1},\kappa_n,V[\dP_{n-2}])& \ast \\ \dot{\dM}^{\oplus}(\kappa_{n-1},\kappa_n, & \kappa_{n+1},V[\dP_{n-1}])\ast \\  & \dot{\dM}^{\oplus}(\kappa_n,\kappa_{n+1},\kappa_{n+2},V[\dP_n])\end{align*}

\item and $\dot{\dP}_\textup{high}^n$ is a $\dP_\textup{low}^n \ast \dot{\dP}_\textup{next}^n$-name for
\[
\seq{\dM^\oplus(\kappa_{k-2},\kappa_{k-1},\kappa_k,V[\dP_{k-2}])}{n+3 \le k<\omega}.
\]
\end{itemize}
	We want to show that $\dP$ forces $\ICNIA(\Theta,\kappa_{n-1})$. Let $G_\textup{low}$ be $\dP_\textup{low}^n$-generic over $V$ and work in $V[G_{\textup{low}}]$. Because $|\dP_\textup{low}^n|<\kappa_n$, $\kappa_n$ remains Mahlo in this model.

Now we need to perform a termspace argument. Here we will use the notation in which $A(\dP_1,\dot{\dP_2})$ is the termspace forcing in which $\dot{\dP_2}$ is the underying forcing and the ordering is taken with respect to what is forced by the empty condition of $\dP_1$ (see \cite[Section 22]{Handbook-Cummings}).

By standard termspace arguments, $\dP_\textup{next}^n\ast\dot{\dP}_\textup{high}^n$ is a projection of $\dP_\textup{next}^n\times A(\dP_\textup{next}^n,\dot{\dP}_\textup{high}^n)$. Since $\dP_\textup{next}^n$ forces $\dot{\dP}_\textup{high}^n$ to be ${<}\,\kappa_n$-strategically closed (using similar arguments to \cite{Cummings-Foreman1998}), $A(\dP_\textup{next}^n,\dot{\dP}_\textup{high}^n)$ is ${<}\,\kappa_n$-strategically closed. Now we focus on $\dP_\textup{next}^n$. Writing $\dM^{\oplus}(\tau,\mu,\kappa,W)$ as $\Add(\tau,\kappa)^W*\dT(\tau,\mu,\kappa)$, we have
\begin{align*}
	\dP_\textup{next}^n & =(\dM^{\oplus}(\kappa_{n-2},\kappa_{n-1},\kappa_n,V[\dP_{n-2}])\times\Add(\kappa_{n-1},\kappa_{n+1})) \\
	& \ast (\dT(\kappa_{n-1},\kappa_n,\kappa_{n+1})*\Add(\kappa_n,\kappa_{n+2})^{V[\dP_n]}*\dT(\kappa_n,\kappa_{n+1},\kappa_{n+2})).
\end{align*}

Let $$\dP_\textup{mid}^n:=\dM^{\oplus}(\kappa_{n-2},\kappa_{n-1},\kappa_n,V[\dP_{n-2}])\times\Add(\kappa_{n-1},\kappa_{n+1})$$
and
\begin{align*}
	\dT_\textup{next}^n & :=\dT(\dM^{\oplus}(\kappa_{n-1},\kappa_n,\kappa_{n+1}))\times A(\dM^{\oplus}(\kappa_{n-2},\kappa_{n-1},\kappa_n),\Add(\kappa_n,\kappa_{n+2})^{V[\dP_n]})\\
	& \times A(\dM^{\oplus}(\kappa_{n-2},\kappa_{n-1},\kappa_n)*\dM^{\oplus}(\kappa_{n-1},\kappa_n,\kappa_{n+1}),\dT(\dM^{\oplus}(\kappa_n,\kappa_{n+1},\kappa_{n+2}))),
\end{align*}
which is ${<}\,\kappa_n$-strategically closed.
Then $\dP_\textup{next}^n$ is easily seen to be a projection of $\dP_\textup{mid}^n\times\dT_\textup{next}^n$.

So in summary, $\dP_\textup{next}^n*\dP_\textup{high}^n$ is a projection of $\dP_\textup{mid}^n\times\dT_\textup{high}^n$, where
$$\dT_\textup{high}^n:=\dT_\textup{next}^n\times A(\dP_\textup{next}^n,\dP_\textup{high}^n).$$

We can consider any extension by $\dP_\textup{mid}^n\times\dT_\textup{high}^n$ as an extension first by $\dT_\textup{high}^n$ and then by $\dP_\textup{mid}^n$. In such an extension, $\ICNIA^+(\Theta,\kappa_{n-1})$ holds: $\dT_\textup{high}^n$ preserves the Mahloness of $\kappa_n$ by its strategic closure and does not add any new conditions to $\dP_\textup{mid}^n$. Ergo, by Lemma \ref{Extension}, $\dP_\textup{mid}^n$ forces $\ICNIA^+(\Theta,\kappa_{n-1})$. Furthermore, any ${<}\,\kappa_n$-sequence added by $\dP_\textup{mid}^n\times\dT_\textup{high}^n$ has been added by $\dP_\textup{mid}^n$, so $\ICNIA^+(\Theta,\kappa_{n-1})$ also holds in any extension by $\dP_\textup{next}^n*\dP_\textup{high}^n$ by Proposition~\autoref{ClubApprDown}.



\begin{mybem} The first author obtained a proof of Theorem \autoref{omega-theorem} using a product rather than an iteration \cite{Jakob-thesis}.\end{mybem}

Here is a question related to the technical aspects of this paper:

\begin{myquest} Suppose $\dP$ is a $\nu^+$-closed forcing and $S \subseteq P_\nu(H(\Theta))$ is a stationary set of internally club sets. Is $S$ stationary in an extension by $\dP$?\end{myquest}


\bibliography{bibliography}

\begin{thebibliography}{10}

\bibitem{Abraham1983}
Uri Abraham.
\newblock Aronszajn trees on {$\aleph _{2}$} and {$\aleph _{3}$}.
\newblock {\em Ann. Pure Appl. Logic}, 24(3):213--230, 1983.

\bibitem{Handbook-Cummings}
James Cummings.
\newblock Iterated forcing and elementary embeddings.
\newblock In Matthew Foreman and Akihiro Kanamori, editors, {\em Handbook of
  Set Theory}, pages 775--883. Springer, 2010.

\bibitem{Cummings-Foreman1998}
James Cummings and Matthew Foreman.
\newblock The tree property.
\newblock {\em Advances in Mathematics}, 133(1):1--32, 1998.

\bibitem{Foreman-Todorcevic2005}
Matthew Foreman and Stevo Todorcevic.
\newblock A new {L}\"{o}wenheim-{S}kolem theorem.
\newblock {\em Trans. Amer. Math. Soc.}, 357(5):1693--1715, 2005.

\bibitem{Fuchino-Rodrigues2018}
Saka{\'e} Fuchino and Andr{\'e} Ottenbreit~Maschio Rodrigues.
\newblock Reflection principles, generic large cardinals, and the continuum
  problem.
\newblock In {\em Symposium on Advances in Mathematical Logic}, pages 1--25.
  Springer, 2018.

\bibitem{Harrington-Shelah1985}
Leo Harrington and Saharon Shelah.
\newblock Some exact equiconsistency results in set theory.
\newblock {\em Notre Dame Journal of Formal Logic}, 26(2):178--188, 1985.

\bibitem{Jakob2023}
Hannes Jakob.
\newblock Disjoint stationary sequences on an interval of cardinals.
\newblock {\em arXiv preprint arXiv:2309.01986}, 2023.

\bibitem{Jakob-thesis}
Hannes Jakob.
\newblock {\em Variants of {M}itchell Forcing}.
\newblock PhD thesis, University of Freiburg, 2023-2025.

\bibitem{Jech2003}
Thomas Jech.
\newblock {\em Set Theory}.
\newblock Springer Monographs in Mathematics. Springer-Verlag, Berlin, the
  third millennium, revised and expanded edition, 2003.

\bibitem{Krueger2007}
John Krueger.
\newblock Internally club and approachable.
\newblock {\em Adv. Math.}, 213(2):734--740, 2007.

\bibitem{Krueger2008b}
John Krueger.
\newblock A general {M}itchell style iteration.
\newblock {\em MLQ Math. Log. Q.}, 54(6):641--651, 2008.

\bibitem{Krueger2009}
John Krueger.
\newblock Some applications of mixed support iterations.
\newblock {\em Ann. Pure Appl. Logic}, 158(1-2):40--57, 2009.

\bibitem{Levine2023a}
Maxwell Levine.
\newblock On disjoint stationary sequences.
\newblock submitted, 2023.

\bibitem{Unger2015}
Spencer Unger.
\newblock Fragility and indestructibility ii.
\newblock {\em Annals of Pure and Applied Logic}, 166(11):1110--1122, 2015.

\bibitem{Usuba2014}
Toshimichi Usuba.
\newblock The approximation property and the chain condition.
\newblock {\em RIMS Kokyuroku}, 1985:130--134, 2014.

\bibitem{Yoshinobu2017}
Yasuo Yoshinobu.
\newblock The $*$-variation of the {B}anach-{M}azur game and forcing axioms.
\newblock {\em Annals of Pure and Applied Logic}, 168(6):1335--1359, 2017.

\end{thebibliography}
\bibliographystyle{plain}

\end{document}